\documentclass[11pt]{article}
\usepackage{amsmath, amsthm}
\usepackage{amscd}
\usepackage{amsfonts}
\usepackage{latexsym}
\usepackage{graphics}
\usepackage{epsf}
\usepackage{epsfig}
\usepackage{vmargin}

\newcommand{\tab}{T_{a,b}}
\newcommand{\R}{\mathbb{R}}

\newcommand{\Z}{\mathbb{Z}}

\makeatletter
\@addtoreset{equation}{section}

\newcommand{\ra}{\rightarrow}

\newtheorem{thm}{Theorem}[section]
\newtheorem{lemma}[thm]{Lemma}
\newtheorem{prop}[thm]{Proposition}
\newtheorem{cor}[thm]{Corollary}

\newtheorem{nameconj}{Conjecture}

\title{Absorbing sets of homogeneous subtractive algorithms}
\author{Tomasz Miernowski \and Arnaldo Nogueira}
\date{\today}

\begin{document}

\maketitle

\begin{abstract}
\noindent
We consider homogeneous  multidimensional continued fraction algorithms, in particular a family of maps which was introduced by F. Schweiger.  We prove his conjecture regarding the existence of an absorbing set for those maps. We also establish that their renormalisations are nonergodic which disproves another conjecture due to Schweiger. Other homogeneous algorithms are also analysed including ones which are ergodic. 
\end{abstract}

\section{Introduction} 

\noindent
Multidimensional continued fraction algorithms have become a classical topic in the Theory of Dynamical Systems. There is a wide literature on the metric theory of those maps and a large collection of examples is available in Schweiger \cite{schweiger}. Besides their number theoretical motivation they also appear in dynamical systems, in particular within the renormalisation theory where Poincar\'e first return maps are considered. A seminal example is the so-called Rauzy induction algorithm of interval exchange transformations \cite{rauzy}  which is a central tool to study the ergodic properties of interval exchanges, e.g.  Veech \cite{veech1,veech2}, Nogueira and Rudolph \cite{arnaldo2} and Avila and Forni \cite{avila}.  The Rauzy algorithm relates one interval exchange to a class of interval exchange maps through a suitable induction process. 

The present work concerns the so-called homogeneous subtractive algorithms.   One of the difficulties that one encounters in studying their dynamics  lays in the fact that they admit an infinite invariant measure. Most of the algorithms which have been studied bare ergodic properties, less attention has been paid to nonergodic algorithms having absorbing set. Homogeneous algorithms having absorbing set are quite common maps though. Here we consider a class of maps which are naturally defined and prove that  they have absorbing set.

Among the well known examples of homogeneous maps we enumerate the Euclidean algorithm, the Brun algorithm and the Selmer algorithm.  Quite often when a nontrivial absorbing set occurs a suitable renormalisation of the algorithm is considered forcing its ergodicity, but very little is known about the dynamical properties of the initial homogeneous algorithm. 

Here our purpose is to study the family of transformations defined as follows: Let $\Lambda^n=\{x=(x_1,x_2,\ldots,x_n) :   0\leq  x_1 \leq \ldots \leq x_n\}$ be the set of those points of $\R^n$ whose coordinates are positive and nondecreasing. Let $a,b\geq 1$ be integers. Given any point $x\in\Lambda^{a+b}$, we define 
\begin{equation} 
\label{tab} T_{a,b}(x)=\pi(x_1,\ldots,x_a,x_{a+1}-x_a,\ldots,x_{a+b}-x_a),
\end{equation}
where $\pi$ is a permutation of the coordinates (depending on the point $x$) which arranges $x_1,\ldots,x_a,x_{a+1}-x_a,\ldots,x_{a+b}-x_a$ in ascending order. The map $T_{a,b}:\Lambda^{a+b}\to\Lambda^{a+b}$ is continuous and piecewise linear.

The family $(1.1)$ is not new and contains several well studied maps given by special values of the parameters $a$ and $b$. The first example, corresponding to $a=b=1$, has the following property: If $x=(x_1,x_2)\in\Lambda^2\cap\Z^2$, then there exists $k\geq 1$ such that $T_{1,1}^k(x)=(0,d)$, where $d\geq 1$ is the greatest common divisor of $x_1$ and $x_2$. That explains why it is called the Euclidean algorithm. It is known that for almost every $x\in\Lambda^2$ the corresponding orbit converges to the origin and the map $T_{1,1}$  is ergodic with respect to Lebesgue measure \cite{veech2,arnaldo1}.

Another special case corresponds to $a\geq 2$  and $b=1$ which is called Brun algorithm. Later it will be shown that for almost every $x\in \Lambda^{a+1}$ the corresponding orbit also converges to the origin. In order to study the metric properties of this algorithm, one may project the dynamics onto the set $B=\{ x\in \Lambda^{a} \ :  \ 0\leq x_1\leq\ldots\leq x_{a}\leq 1\}$, which is of finite Lebesgue measure. The new map $S_{a}:B\to B$ is defined through the commutative diagram
$$
\begin{CD}
\mbox{$\Lambda^{a+1}$} @>\mbox{$T_{a,1}$}>> \mbox{$\Lambda^{a+1}$}\\
@VpVV @VVpV\\
\mbox{$B$} @>>\mbox{$S_{a,1}$}> \mbox{$B$}
\end{CD}
$$
where $p$ is the projection $(x_1,x_2,\ldots,x_{a+1})\mapsto (x_1/x_{a+1},x_2/x_{a+1},\ldots,x_{a}/x_{a+1})$.  A result by Schweiger (\cite{schweiger}, p.50) states that $S_{a,1}$ is ergodic with respect to Lebesgue measure. Nevertheless the question of ergodicity of the original map $T_{a,1}:\Lambda^{a+1}\to\Lambda^{a+1}$ remains open.

The third type of map corresponds to $a=1$ and $b\geq 2$ and was considered by Meester and Nowicki \cite{meester-nowicki} and later by Kraaikamp and Meester \cite{kraaikamp-meester}. The map appeared naturally in their study of a parametric percolation model on the lattice $\Z^{b+1}$. They proposed an algorithm for computation of the critical probability of the model. The efficiency of this algorithm may be expressed through the absorbing properties of the set 
$$ 
A=\{ x\in\Lambda^{b+1} \ : \ x_1+x_2+\ldots+x_{b+1}\leq bx_{b+1}\} 
$$ 
under the dynamics of $T=T_{1,b}$. We recall that $T$ is an {\it absorbing set} if almost surely
$$
\bigcup_{n\in \Z}T^n(A) = \Lambda^{b+1}.
$$
To be more precise, if $A$ is absorbing, then the algorithm gives the critical probability of the model in a finite number of steps, for almost every choice of model parameters. The case $b=2$ was first studied in \cite{meester-nowicki}. In \cite{kraaikamp-meester} the results were extended to the case $b\geq 2$.

The general case of the map $T_{a,b}:\Lambda^{a+b}\to\Lambda^{a+b}$ was first studied by Schweiger (see \cite{schweiger}, chapter 9). He noticed that for every $a,b\geq 1$, the corresponding set 
$$
A=\{x\in\Lambda^{a+b} \ : \ x_1+x_2+\ldots + x_{a+b} \leq b x_{a+b}\}
$$ 
was $T_{a,b}$- forward invariant and conjectured that this set should be absorbing under $T_{a,b}$. Moreover, in the case $b\geq 2$, he defined the smaller invariant set 
$$
D=\{x\in\Lambda^{a+b} \ : \ x_1+x_2+\ldots + x_{a+1} \leq  x_{a+2}\}
$$ 
and conjectured that it was absorbing as well.

\begin{nameconj}[Schweiger] Let $a\geq 1$ and $b\geq 2$. Then the corresponding set $D$ is absorbing for $T_{a,b}$. \end{nameconj}

In order to prove this conjecture, we first establish an extension of the Kraaikamp-Meester-Nowicki theorem.

\begin{thm} \label{A} Let $a\geq 1$ and $b\geq 2$. The set $A$ is absorbing for $T_{a,b}$.  \end{thm}

Analysing the limiting behavior of the orbits, we  prove a stronger claim.

\begin{thm} \label{limit} 
Let $a\geq 1$ and $b\geq 2$. Then, for almost every $x\in\Lambda^{a+b}$,  
$$\lim_{k\to\infty} T^k_{a,b}(x)=(0,\ldots,0,x_{a+2}^{\infty},\ldots,x_{a+b}^{\infty}),$$
where $0<x_{a+2}^{\infty} \leq \ldots\leq x_{a+b}^{\infty}$. 
\end{thm} 

As an immediate corollary, we get that Conjecture 1 holds.
\\

Schweiger also attempted to study the ergodic properties of the map $T_{a,b}$. To this end, he used the same projection as in the case of the Brun algorithm. The new transformation $S_{a,b}:B\to B$ is defined through the commutative diagram
$$
\begin{CD}
\mbox{$\Lambda^{a+b}$} @>\mbox{$\tab$}>> \mbox{$\Lambda^{a+b}$}\\
@VpVV @VVpV\\
\mbox{$B$} @>>\mbox{$S_{a,b}$}> \mbox{$B$}
\end{CD}
$$
where $B=\{ x\in \Lambda^{a+b-1} \ :  \ 0\leq x_1\leq\ldots\leq x_{a+b-1}\leq 1\}$ and $p$ is the projection $(x_1,x_2,\ldots,x_{a+b})\mapsto (x_1/x_{a+b},x_2/x_{a+b},\ldots,x_{a+b-1}/x_{a+b})$. Let $D$ stand also for the projection of the set $D\subset \Lambda^{a+b}$ onto $B$. This projected set is $S_{a,b}$-forward invariant, just as the original $D$ is for $T_{a,b}$. The transformation $S_{a,b}$ restricted to the set $D$ admits an invariant measure which is absolutely continuous with respect to Lebesgue measure and whose density was explicitely calculated by Schweiger (\cite{schweiger}, p.80). Having this in mind, it is natural to ask about the ergodic properties of $S_{a,b}$ with respect to Lebesgue measure on $D$. The following conjecture was formulated.

\begin{nameconj}[Schweiger] Let $a\geq 1$ and $b\geq 2$. Then the restriction of the map $S_{a,b}$ to the set $D$ is ergodic with respect to Lebesgue measure on $D$.  \end{nameconj} 

This conjecture is more subtle. First, we prove the following theorem that confirms the conjecture in the case $a=1$ and $b=2$. 

\begin{thm} \label{erg} The map $S_{1,2}:B\to B$ is totally dissipative and ergodic with respect to Lebesgue measure. \end{thm}

On the other hand, from Theorem \ref{limit} we deduce that for almost every $x\in D$ we have $\lim_{k\to\infty}S_{a,b}^k(x)=(0,\ldots,0,x_{a+2}^{\infty},\ldots, x_{a+b-1}^{\infty})$. If $b\geq 3$, we may show that $f(x)=x_{a+b-1}^{\infty}$ is a nonconstant invariant function for $S_{a,b}$. This disproves the second Schweiger conjecture in this case.  

\begin{thm} \label{noterg} For $b\geq 3$, the map $S_{a,b}$ is not ergodic with respect to Lebesgue measure. \end{thm}

The same argument shows that for any $a\geq 1$ and $b\geq 2$ the nonprojected algorithm $T_{a,b}:\Lambda^{a+b}\to\Lambda^{a+b}$ is not ergodic with respect to Lebesgue measure. The question of the ergodicity of the projected transformation $S_{a,2}$ for $a\geq 2$ seems more difficult and is not answered here. 

The paper is organized as follows. The next section regards a transformation of $\R^3_+$ which is closely related to the map $T_{1,2}$ of our family. We prove the first Schweiger conjecture in this special case. The Section 3 is devoted to some general remarks concerning the family of transformations $\tab$. Some notations used throughout the text are also introduced. In Section 4, which is highly inspired by the work of \cite{kraaikamp-meester}, we characterize the invariant sets $A$ and $D$ through the limiting behaviour of the orbits. In order to show that the set $A$ is absorbing, we have to study the dynamics of $\tab$ on its complement which is done in Section 5. In particular, we define some smaller set $\Theta$ contained in the complement of $A$ and define a kind of first return map induced by $\tab$ on this set. This is done despite the infinite measure of the sets involved.  A family of partitions of $\Theta$, connected to the iterations of the first return map, is defined. In Section 6, we project the dynamics onto a finite Lebesgue measure set of lower dimension which enables us to study  metric properties of the partitions mentioned before. Those properties are used in Section 7 to prove Theorems \ref{A} - \ref{noterg}. In the last section we define other families of subtractive algorithms and discuss the possibility of extending our results to those new transformations.  
\\

After we concluded the present paper we have been informed that an independent article \cite{fkn} with similar results (Theorem \ref{limit} and \ref{noterg}) is currently accepted for publication by Israel Journal of Mathematics. However, we would like to mention that our paper presents a different approach to the subject and places it in a larger perspective. In particular it allows us to prove Theorem \ref{erg}.


\section{An example}

The present section deals with a particular version of the map $T_{1,2}$ and is independent of the remaining of the paper. It illustrates our geometric approach to the proof of Conjecture 1, which will further be extended to the general case of the map $T_{a,b}$.   

Let  $T:\R^3_+\to\R^3_+$ be given by
$$T(x_1,x_2,x_3)=\left\{\begin{array}{ll} (x_1,x_2-x_1,x_3-x_1) & \text{if} \ \min\{x_1,x_2,x_3\}=x_1  \\ (x_1-x_2,x_2,x_3-x_2) & \text{if} \ \min\{x_1,x_2,x_3\}=x_2 \\ (x_1-x_3,x_2-x_3,x_3) & \text{if} \ \min\{x_1,x_2,x_3\}=x_3 \end{array} \right.$$ 
which is well defined for almost every $x\in\R^3_+$. We will show that the set 
$$A=\{(x_1,x_2,x_3)\in\R^3_+ : x_1+x_2+x_3\leq 2\max\{x_1,x_2,x_3\} \}$$
is absorbing for $T$ which will clearly imply the Conjecture 1 for the map $T_{1,2}$. An argument similar to the one presented below, due to J.-C. Yoccoz, can be found in \cite{levitt}.

Let us introduce the following notation. Let $\mathcal{F}=(f_1,f_2,f_3)$ be a basis of $\R^3$. We define the corresponding positive simplicial cone
$$
\R^3_+(\mathcal{F}) = \{x\in \R^3_+ : x=x_1f_1+x_2f_2+x_3f_3, \ x_1,x_2,x_3\geq 0\} $$
and the analogue of the set $A$ relative to this new basis 
$$A(\mathcal{F}) = \{ x=x_1f_1+x_2f_2+x_3f_3 : (x_1,x_2,x_3)\in A\}.
$$
The set $A(\mathcal{F})$ is equal to the union of three cones $\R^3_+(f_1,f_1+f_2,f_1+f_3)$, $\R^3_+(f_2,f_1+f_2,f_2+f_3)$ and $\R^3_+(f_3,f_1+f_3,f_2+f_3)$. 

A simplicial cone $\R^3_+(f_1,f_2,f_3)\subset \R^3_+$ may be visualized as its intersection with the $2$-dimensional simplex $\Delta=\{x\in\R^3_+ : x_1+x_2+x_3=1\}$. The cone is then represented by the triangle whose vertices are the radial projections of vectors $f_1$, $f_2$ and $f_3$ onto $\Delta$. For convenience we choose to label these vertices with the original coordinates of the vectors $f_i$ rather then their projections.    

In particular, the original set $A$ may be seen as $A(e_1,e_2,e_3)$, where $(e_1,e_2,e_3)$  stands for the canonical basis of $\R^3$. We get that $A$ is a union of three simplicial cones 
$$ A = A(e_1,e_2,e_3) = \R^3_+(e_1,e_1+e_2,e_1+e_3) \cup \R^3_+(e_2,e_1+e_2,e_2+e_3) \cup \R^3_+(e_3,e_1+e_3,e_2+e_3).$$
Its image in the simplex $\Delta$ is the union of three shaded triangles in Figure 1. 

 Let now $x=(x_1,x_2,x_3)$ represent the vector $x$ in the canonical basis of $\R^3$. Consider the cone $\R^3_+(e_1+e_2+e_3,e_2,e_3)$ corresponding to the set of vectors such that $x_1\leq x_2,x_3$. We remark that
$$ x=x_1e_1+x_2e_2+x_3e_3=x_1(e_1+e_2+e_3)+(x_2-x_1)e_2+(x_3-x_1)e_3.$$
This relation means that the map $T$ may be seen as a convenient basis change in $\R^3$ and implies
\begin{equation} \label{relation} T^{-1}(A)\cap \R^3_+(e_1+e_2+e_3,e_2,e_3)=A(e_1+e_2+e_3,e_2,e_3).\end{equation}
The projected version of this relation is depicted in Figure 2. The analogous relations hold in $\R^3_+(e_1+e_2+e_3,e_1,e_3)$ and $\R^3_+(e_1+e_2+e_3,e_1,e_2)$ where $x_2\leq x_1,x_3$ and $x_3\leq x_1,x_2$ respectively. Putting them together yields
$$T^{-1}(A)=A\cup\R^3_+(2e_1+e_2+e_3, e_1+2e_2+e_3,e_1+e_2+2e_3).$$ 

\begin{figure}[h]
\begin{picture}(300,150)(0,5)

\put(20 ,30){\mbox{\includegraphics[scale=0.6]{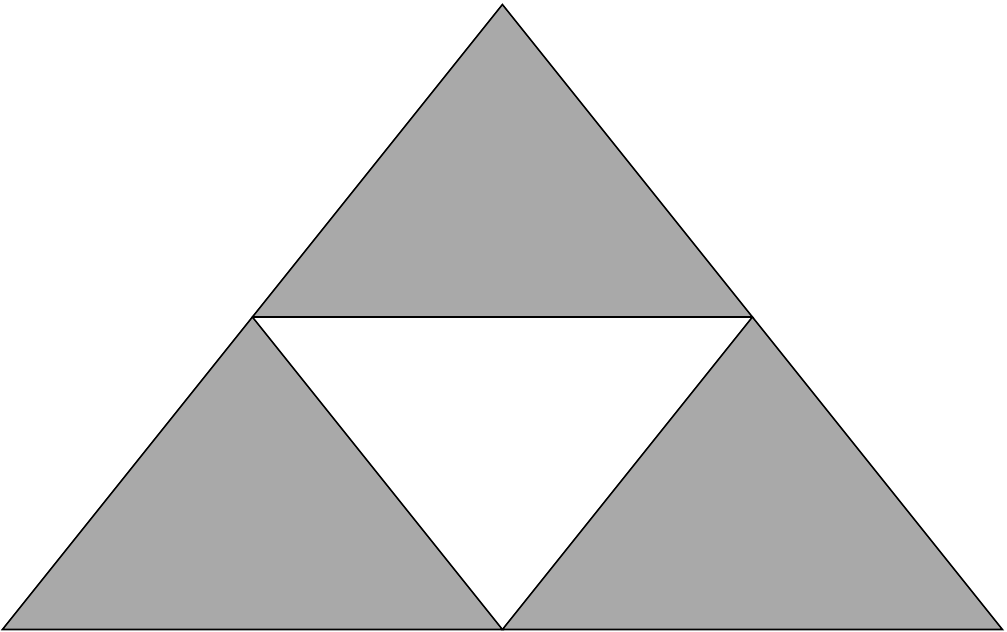}}}

\put(240,30){\mbox{\includegraphics[scale=0.6]{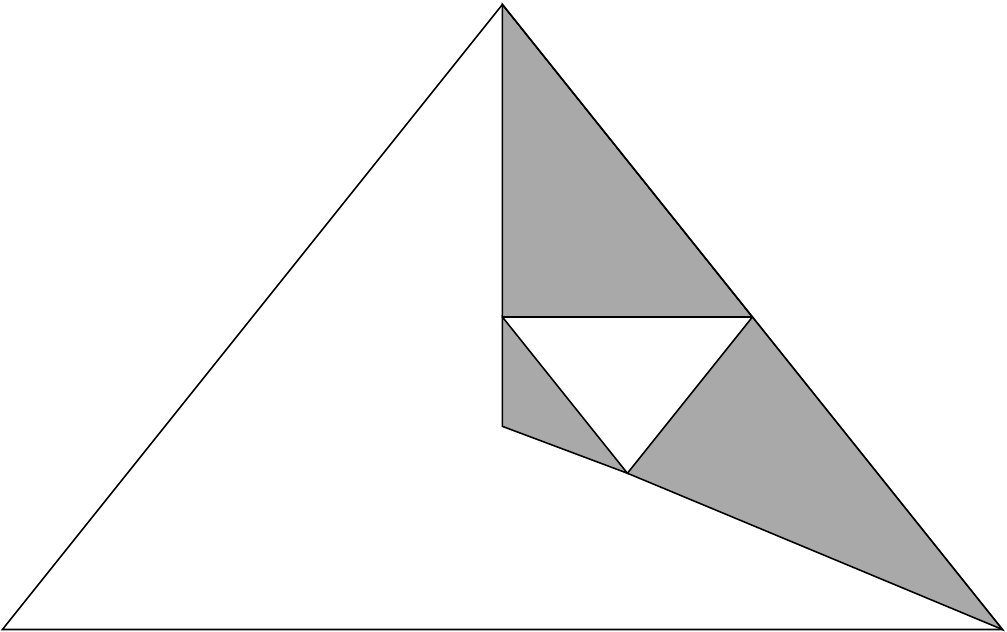}}}

\put(5,20){\mbox{$(1,0,0)$}}
\put(180,20){\mbox{$(0,1,0)$}}
\put(80,0){\mbox{Figure 1}}
\put(310,0){\mbox{Figure 2}}
\put(90,145){\mbox{$(0,0,1)$}}
\put(90,20){\mbox{$(\frac{1}{2},\frac{1}{2},0)$}}
\put(20,85){\mbox{$(\frac{1}{2},0,\frac{1}{2})$}}
\put(155,85){\mbox{$(0,\frac{1}{2},\frac{1}{2})$}}

\put(310,145){\mbox{$(0,0,1)$}}
\put(390,20){\mbox{$(0,1,0)$}}
\put(288,63){\mbox{$(\frac{1}{3},\frac{1}{3},\frac{1}{3})$}}
\put(375,85){\mbox{$(0,\frac{1}{2},\frac{1}{2})$}}
\put(288,83){\mbox{$(\frac{1}{4},\frac{1}{4},\frac{1}{2})$}}
\put(320,45){\mbox{$(\frac{1}{4},\frac{1}{2},\frac{1}{4})$}}
\end{picture}

\end{figure}

\bigskip

In order to study the second inverse image of $T$, we take the cone $\R^3_+(e_1+e_2+e_3,e_2,e_3)$ and divide it into three sub-cones $\R^3_+(e_1+2e_2+2e_3,e_2,e_3)$, $\R^3_+(e_1+2e_2+2e_3,e_1+e_2+e_3,e_3)$ and $\R^3_+(e_1+2e_2+2e_3,e_1+e_2+e_3, e_2)$, which correspond to the choice of the smallest coordinate with respect to the basis $(e_1+e_2+e_3,e_2,e_3)$. We get the following counterpart of the relation (\ref{relation}) 

$$ T^{-2}(A) \cap  \R^3_+(e_1+2e_2+2e_3,e_2,e_3)=A(e_1+2e_2+2e_3,e_2,e_3)$$
and the other two corresponding to the remaining two sub-cones. 

Arguing by induction, one may show that in the $k$th step we obtain the decomposition of $\R^3_+$ into $3^k$ simplicial cones, with disjoint non empty interiors. If $\R^3_+(f_1,f_2,f_3)$ is a cone of the $k$th decomposition, then    
$$T^{-k}(A)\cap \R^3_+(f_1,f_2,f_3)=A(f_1,f_2,f_3).$$
 In the next step, we decompose $\R^3_+(f_1,f_2,f_3)$ into three sub-cones corresponding to three basis $(f_1+f_2+f_3,f_1,f_2)$, $(f_1+f_2+f_3,f_2,f_3)$ and $(f_1+f_2+f_3,f_1,f_3)$.
 
 We may also describe this construction by putting the stress on the complement of the inverse images of $A$. For every $k\geq 0$, the complement in $\R^3_+$ of the set $T^{-k}(A)$ is composed of $3^k$ simplicial cones of disjoint interiors. If $\R^3_+(f_1,f_2,f_3)$ is one of them, then 
 \begin
 {equation} \label{middle} T^{-(k+1)}(A)\cap\R^3_+(f_1,f_2,f_3)=\R^3_+(f_1+f_2,f_2+f_3,f_1+f_3).
 \end{equation}

 \begin{figure}[h]
\begin{picture}(300,150)(0,10)

\put(5,20){\mbox{\includegraphics[scale=0.4]{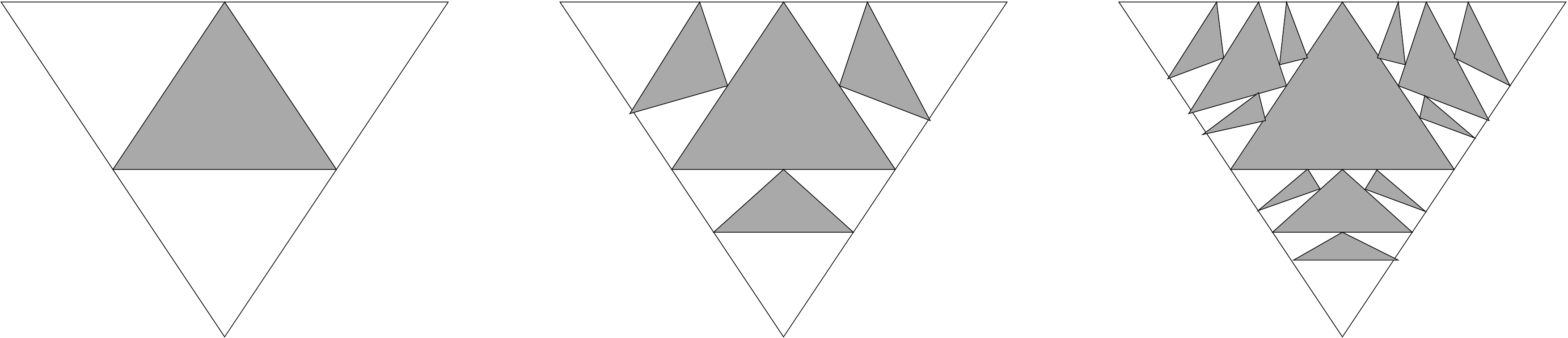}}}



\put(45,10){ \mbox{\footnotesize $(\frac{1}{2},\frac{1}{2},0)$}}
\put(180,0){ \mbox{Figure 3}}
\put(-10,115){ \mbox{\footnotesize $(\frac{1}{2},0,\frac{1}{2})$}}
\put(100,115){ \mbox{\footnotesize $(0,\frac{1}{2},\frac{1}{2})$}}

\put(45,115){ \mbox{\footnotesize $(\frac{1}{4},\frac{1}{4},\frac{1}{2})$}}
\put(-5,60){ \mbox{\footnotesize $(\frac{1}{2},\frac{1}{4},\frac{1}{4})$}}
\put(90,60){ \mbox{\footnotesize $(\frac{1}{4},\frac{1}{2},\frac{1}{4})$}}

\put(160,115){ \mbox{\footnotesize $(\frac{1}{3},\frac{1}{6},\frac{1}{2})$}}
\put(210,115){ \mbox{\footnotesize $(\frac{1}{6},\frac{1}{3},\frac{1}{2})$}}

\put(153,40){ \mbox{\footnotesize $(\frac{1}{2},\frac{1}{3},\frac{1}{6})$}}
\put(225,40){ \mbox{\footnotesize $(\frac{1}{3},\frac{1}{2},\frac{1}{6})$}}

\put(130,75){ \mbox{\footnotesize $(\frac{1}{2},\frac{1}{6},\frac{1}{3})$}}
\put(246,75){ \mbox{\footnotesize $(\frac{1}{6},\frac{1}{2},\frac{1}{3})$}}

\end{picture}\\\

\end{figure}

Figure 3 depicts the first three steps of this construction projected onto $\Delta$ and restricted to the "middle" triangle of vertices $(\frac{1}{2},\frac{1}{2},0)$, $(0,\frac{1}{2},\frac{1}{2})$ and  $(\frac{1}{2},0,\frac{1}{2})$. 
 
 In order to show that the set $A$ is absorbing, we have to show that, for every step $k\geq 0$, the cone $\R^3_+(f_1+f_2,f_2+f_3,f_1+f_3)$ corresponds to a large part of the volume of the cone $\R^3_+(f_1,f_2,f_3)$ in (\ref{middle}). Since both cones are of infinite measure, we compare the areas of their images in $\Delta$. Let $\Delta_{f_1 f_2 f_3}$ be the triangle in $\Delta$ corresponding to the cone $\R^3_+(f_1,f_2,f_3)$. It may be shown (see \cite{arnaldo5}, Lemma 3.2) that  
$$Area(\Delta_{f_1 ,f_2 ,f_3})=\frac{\sqrt{3}}{2}\frac{\vert \det( M_{f_1, f_2 ,f_3})\vert }{\Vert f_1\Vert_1 \Vert f_2\Vert_1 \Vert f_3\Vert_1},$$
where $M_{f_1 ,f_2, f_3}$ is the non negative matrix whose columns are given by the coordinates of the vectors  $f_1$, $f_2$ and $f_3$.

Let $k\geq 0$ and $\Delta_{f_1, f_2, f_3}$ be one of the triangles in the complement of $T^{-k}(A)$. By induction we show that any of the mutual ratios of the numbers $ \Vert f_1\Vert_1, \Vert f_2\Vert_1, \Vert f_3\Vert_1$ does not exceed $k+1$. It is clear in the case of the unique triangle in the complement of $A$ ($k=0$). Assume the assertion is true for the triangle $\Delta_{f_1, f_2, f_3}$ of the complement of $T^{-k}(A)$. This implies the assertion for the three triangles $\Delta_{f_1,f_1+f_2,f_1+f_3}$, $\Delta_{f_2,f_1+f_2,f_2+f_3}$, $\Delta_{f_3,f_1+f_3,f_2+f_3}$  of the complement of $T^{-(k+1)}(A)$. Indeed, for $\Delta_{f_1,f_1+f_2,f_1+f_3}$ we get

$$\frac{\Vert f_1+f_2\Vert_1}{\Vert f_1 + f_3\Vert_1} \leq \frac{\Vert f_1+f_2\Vert_1}{\Vert f_1\Vert_1}=1+\frac{\Vert f_2\Vert_1}{\Vert f_1\Vert_1}\leq k+2.$$

Applying this estimation we get
$$
\frac{Area(\Delta_{f_1+f_2,f_2+f_3,f_1+f_3})}{Area(\Delta_{f_1,f_2,f_3})} =  \frac{2\Vert f_1\Vert_1 \Vert f_2\Vert_1 \Vert f_3\Vert_1}{ \Vert f_1+f_2\Vert_1 \Vert f_2+f_3\Vert_1 \Vert f_1+f_3\Vert_1} \geq \frac{1}{2(k+2)}.
$$
This together with (\ref{middle}) implies that 
$${Area}(\Delta \setminus  T^{-(k+1)}(A) ) \leq  \left(1-\frac{1}{2(k+2)}\right) {Area} (\Delta \setminus  T^{-k}(A) )$$
for every $k\geq 0$. As the series $\sum_{k\geq 0}\frac{1}{2(k+2)}$ diverges, the set $A$ is absorbing for the map $T$.


\section{Some notations and remarks}      

Throughout the remaining part of this paper the following notations are used. All elements of $\R^n$ are given in row notation. The canonical basis of $\R^n$ is denoted by $(e_1,e_2,\ldots,e_n)$ and all vectors of $\R^n$ are always expressed with respect to this basis. Accordingly, the expression $x=(x_1,x_2,\ldots,x_n)$ represents the vector $x$  in the canonical basis. Where matrix multiplications are used, the symbol $x^T$ stands for the transpose of $x$, i.e. the column vector corresponding to $x$. We also define $\sigma(x)$ to be the sum of the coordinates of a vector $x$, that is if $x\in\R^n$ then $\sigma(x)=x_1+\ldots+x_n$. The actual value of $n$ will be clear from the context. 

While studying a transformation $T:\R^n\to\R^n$ we will often write $x^{(k)}$ instead of $T^k(x)$, for $x\in\R^n$ and $k\geq 0$. Accordingly, for any $i=1,\ldots,n$, the symbol $x^{(k)}_i$ will stand for the $i$th coordinate of the point $x^{(k)}$.   

When speaking about a measure, we always mean Lebesgue measure in the corresponding $\R^n$ space. Thus all "almost surely" and "almost every" statements refer to the corresponding Lebesgue measure. 
\\

Let $a,b\geq 1$ be integers. Let $x\in\Lambda^{a+b}$ and define $y\in\R_+^{a+b}$ by
$$y=(x_1,\ldots,x_a,x_{a+1}-x_a,\ldots,x_{a+b}-x_a).$$
For almost every $x$ there exists a unique permutation $\pi_x$ of the set $\{1,2,\ldots,a+b\}$ such that $(y_{\pi_x(1)},y_{\pi_x(2)},\ldots,y_{\pi_x(a+b)})\in\Lambda^{a+b}$. The only ambiguity appears when two or more coordinates of $y$ are equal. In such a case, we define $\pi_x$ to be the only permutation which satisfies $\pi(i)<\pi(j)$ anytime $y_i=y_j$ and $i<j$.  

We may thus redefine the map (\ref{tab}) as the transformation $\tab:\Lambda^{a+b}\to\Lambda^{a+b}$ defined by 
\begin{equation} \label{T}  \tab(x)=(y_{\pi_x(1)},y_{\pi_x(2)},\ldots,y_{\pi_x(a+b)}). \end{equation}
The map $\tab$ is continuous and piecewise linear, where the number of pieces corresponds to the number of different permutations $\pi_x$ involved. Since for every $x\in\Lambda^{a+b}$ we have  $$x_1\leq  x_2\leq \ldots\leq x_a \ \ \ \text{and} \ \ \ \ x_{a+1}-x_a\leq x_{a+2}-x_a\leq \ldots\leq x_{a+b}-x_a,$$ every permutation $\pi_x$ satisfies $\pi_x(1)<\ldots<\pi_x(a)$ and $\pi_x(a+1)<\ldots<\pi_x(a+b)$. The number of permutations is thus equal to the number of different shuffles of a deck of $a$ cards with a deck of $b$ cards, namely $\binom{a+b}{a}$. 

Let $\Pi_{a,b}$ be the set of all these permutations. For every $\pi\in\Pi_{a,b}$ we define the cylinder $\Lambda_{\pi}=\{x\in\Lambda^{a+b}:\pi_x=\pi\}$. The cylinders $\Lambda_{\pi}$ define a partition of the set $\Lambda^{a+b}$ into $\binom{a+b}{a}$ regions, with nonempty and disjoint interiors, which intersect  along null Lebesgue measure sets. We also define a family $\{L_{\pi}, \pi\in\Pi_{a,b}\}$ of matrices in $GL(a+b,\Z)$ such that $T_{a,b}$ restricted to $\Lambda_{\pi}$ is given by $x\mapsto L_{\pi} x^T$. From (\ref{T}) we easily deduce that the line vectors of the matrix $L_{\pi}$ are 
\begin{equation} \label{lines}  e_1,  e_2, \ldots, e_{a-1}, e_a,  e_{a+1}-e_a,\ldots, e_{a+b}-e_a,  \end{equation} 
rearranged in the order given by the corresponding permutation $\pi$. To be more precise, if $L_j$ stands for the $j$th line of the matrix $L_{\pi}$, we have $L_{\pi(i)}={e_i}$ for $i=1,\ldots,a$ and $L_{\pi(i)}=e_i-e_a$ for $i=a+1,\ldots,a+b$.

The map $T_{a,b}$ is not globally injective. The next lemma comes from \cite{schweiger}. 

\begin{lemma}[\cite{schweiger}, p.69]  Every cylinder $\Lambda _{\pi}$ is full, that is $T_{a,b}(\Lambda_{\pi})=\Lambda^{a+b}$. \end{lemma}

We may also consider the family of inverse matrices $\{M_{\pi}, \pi\in\Pi_{a,b}\}$, where $M_{\pi}=L_{\pi}^{-1}$, corresponding to all inverse branches $M_{\pi}:\Lambda^{a+b}\to \Lambda_{\pi}$ of the transformation $\tab$. For every permutation $\pi$ the corresponding matrix $M_{\pi}$ is a non-negative element of  $GL(a+b,\Z)$. From (\ref{lines}) we get that the column vectors of $M_{\pi}$ are 
\begin{equation} \label{columns} e_1 ^T,\ldots, e_{a-1} ^T, (e_a+e_{a+1}+\ldots+ e_{a+b})^T, e_{a+1} ^T,\ldots,e_{a+b} ^T, \end{equation}
arranged once again in the order given by $\pi$. If $C_j$ stands for the $j$-th column of $M_{\pi}$, we have $C_{\pi(i)}=e_i ^T$ for $i\neq a$ and $C_{\pi(a)}=(e_a+e_{a+1}+\ldots+e_{a+b})^T$. 

We are interested in the absorbing properties of the following sets
$$ A= \{x  \in \Lambda^{a+b} \ : \   \sigma(x)\leq bx_{a+b} \} \  \text{and} \ D=\{x\in\Lambda^{a+b} \ : \ x_1+x_2 +\ldots+x_{a+1}\leq x_{a+2}\},$$ 
the former being defined for all $a,b\geq 1$ and the latter for $b\geq 2$. In the case $b=1$, the set $A$ is a null measure set. For $b\geq 2$, both sets $A$ and $D$ are infinite measure sets. Moreover $D\subset A$ with $D=A$ if and only if $b=2$. We recall the following result due to Schweiger: 

\begin{lemma}[\cite{schweiger}, p.72 and 78]  The sets $A$ and $D$ are forward invariant for $T_{a,b}$, i.e. $T_{a,b}(A)=A$ and $T_{a,b}(D)=D$. \end{lemma}

Our goal is to show that they are actually absorbing for $T_{a,b}$. 


\section{Limiting behavior of orbits}
 
It is a trivial remark that the orbit of any point $x\in\Lambda^{a+b}$ under the action of $T_{a,b}$ is convergent in $\Lambda^{a+b}$. Indeed, for any $i=1,\ldots,a+b$ fixed, the corresponding coordinate sequence $(x^{(k)}_i)_{k\geq 0}$, is positive and nonincreasing. We will show that the set $A$ attracts almost every point $x\in\Lambda^{a+b}$ for which $x^{(k)}$ do not converge to the origin as $k\to\infty$. Equivalently, almost every point whose entire orbit stays outside of $A$ is characterized by the condition  $x^{(k)}\to (0,\ldots,0)$. In the case $a=1$ and $b\geq 2$ this was shown in \cite{kraaikamp-meester}.

\begin{thm}[\cite{kraaikamp-meester}, Theorem 1 and 2] Let $b\geq 2$. For almost every $x\in\Lambda^{b+1}$ the sequence $T_{1,b}^k(x)$, $k\geq 1$, does not converge to the origin and the set $A$ is absorbing under $T_{1,b}$. \end{thm}

We begin with results that illustrate the dynamical behavior of homogeneous algorithms.

\begin{lemma} Let $a\geq 1$ and $b\geq 1$. For almost every $x\in\Lambda^{a+b}$ we have  
$$\lim_{k\to+\infty}\tab^k(x)=(x_1^{\infty},x_{2}^{\infty},\ldots,x_{a+b}^{\infty}),$$
with $x_{1}^{\infty}=x_2^{\infty}= \ldots = x_{a+1}^{\infty} = 0$.
 \end{lemma}
 
\begin{proof} Let $x\in\Lambda^{a+b}$. For every $i=1,\ldots,a+b$, the sequence $(x_i^{(k)})_{k\geq 0}$, is convergent to some $x_i^{\infty}$. We obviously have $0\leq x_1^{\infty}\leq x_2^{\infty}\leq \ldots\leq x_{a+b}^{\infty}$ since at every iteration the coordinates are reordered to be nondecreasing.  

Suppose that the coordinates of the point $x$ are rationally independent. This implies that $x_i^{(k)} > 0$ for all $i=1,\ldots,a+b$ and $k\geq 0$. For every $k\geq 1$ there exists $m\geq 1$ such that $x_{a+1}^{(k)}-mx_a^{(k)}\leq x_a^{(k)}$. If $m$ is the smallest one for which this inequality is satisfied, we get $x_{a+1}^{(k+m)} = x_a^{(k)}$ and the two limits $x_{a}^{\infty}$ and $x_{a+1}^{\infty}$ must be equal. 

Since $x_a^{(k)}$ and $x_{a+1}^{(k)}$ converge to the same limit, for every $\varepsilon >0$ there exists  $k\geq 1$ such that $0\leq x_{a+1}^{(k)}-x_a^{(k)} \leq \varepsilon$. This implies that $x_1^{(k+1)}=\min\{ x_1^{(k)}, x_{a+1}^{(k)}-x_a^{(k)}\}\leq \varepsilon$. We get $x_1^{\infty}=0$. 

Suppose that $x_i^{\infty}>0$ for some $2\leq i\leq a+1$ and let $j=\min\{ 2\leq i\leq a+1 : x_i^{\infty}>0\}$. Since $x_{j-1}^{\infty}=0$, there exists some $k\geq 1$ such that $x_{j-1}^{(k)}\leq x_j^{\infty}/2$ and $x_{a+1}^{(k)}-x_a^{(k)}\leq x_j^{\infty}/2$. This implies $x_j^{(k+1)}=\max\{x_{j-1}^{(k)},x_{a+1}^{(k)}-x_a^{(k)}\}\leq x_j^{\infty}/2$, which is in contradiction with the fact that the sequence $x_j^{(k)}$ is nonincreasing. We get $x_i^{\infty}=0$ for every $i=1,\ldots,a+1$. \end{proof}

Another useful result states the following.

\begin{lemma} \label{sum} Let $a\geq 1$ and $b\geq 1$. Let $x\in\Lambda^{a+b}$ satisfy $x^{(k)}\to 0$ as $k\to +\infty$. Then $\sigma(x)=b \sum_{k\geq 0} x_a^{(k)}$ and thus the sum involved is finite for almost every $x$.  \end{lemma}

\begin{proof} For any $m\geq 1$ we clearly have $\sigma(x^{(m)})=\sigma(x^{(m-1)})-bx_a^{(m-1)}$ which gives $$\sigma(x^{(m)})=\sigma(x)-b\sum_{k=0}^{m-1}x_a^{(k)}.$$ Since $\sigma(x^{(m)})$ goes to zero as $m$ goes to infinity, we get the claim. \end{proof}

The last two lemmas applied to the Brun algorithm $T_{a,1}:\Lambda^{a+1}\to\Lambda^{a+1}$ give 

\begin{cor} \label{brun} Let $a\geq 2$, $b=1$ and write $T_{a,1}^k(x)=x^{(k)}$. Then for almost every $x\in\Lambda^{a+1}$ we get
$$\lim_{k\to +\infty} x^{(k)}=0 \ \ \  \text{and}  \ \ \ \sigma(x)=\sum_{k=1}^{\infty}x_a^{(k)}.$$ \end{cor}

We will now describe the generic limiting behavior in the set $A$.

\begin{lemma} \label{sigma+} Let $a\geq 1$ and $b\geq 2$. Then  $\lim_{k\to\infty} \tab^k(x)\neq (0,\ldots,0)$ for almost every $x\in A$. 
\end{lemma}

\begin{proof} First, remark that $$x_{a+b}^{(k+1)}=\max\{x_a^{(k)},x_{a+b}^{(k)}-x_a^{(k)}\}\geq x_{a+b}^{(k)}-x_a^{(k)}$$ for every $k\geq 0$. This implies that $x_{a+b}^{(k)}\geq x_{a+b}-\sum_{i=0}^{k-1} x_a^{(i)}$. Now suppose that $x^{(k)}\ra 0$ as $k\ra \infty$. From Lemma \ref{sum} we get $\sigma(x)=b\sum_{k\geq 0} x^{(k)}_a$. If $x\in A$, we also have
$x_{a+b}\geq \sigma(x)/b=\sum_{k\geq 0} x^{(k)}_a$. Since $x^{(k)}\ra 0$, we must have $$0 < x_{a+b} - \sum_{i=0}^{k-1} x^{(i)}_a\leq x_{a+b}^{(k)}\to 0.$$ This implies $x_{a+b}= \sigma(x)/b$ which is a null measure condition. \end{proof}

The proof of the next proposition is a straightforward generalization of an argument of  \cite{kraaikamp-meester}.

\begin{prop} \label{origin} Let $a\geq 1$ and $b\geq 2$. Assume that the coordinates of $x\in\Lambda^{a+b}$ are rationally independent and suppose that $x^{(k)}\not\to (0,\ldots,0)$ as $k\to \infty$. Then there exists $k\geq 1$ such that $x^{(k)}\in A$. 
\end{prop}

\begin{proof} Let $x$ be such that $\lim_{k\to\infty} x^{(k)} = (0,\ldots,0,x_{a+2}^{\infty},\ldots,x_{a+b}^{\infty})\neq (0,\ldots,0)$ and let $m=\min \{i : x_{a+i}^{\infty}\neq 0\}$. Obviously $2\leq m\leq b$. Since all the coordinates to the left of $x_{a+m}$ go to zero as $k\to\infty$, there exists $k\geq 1$ such that 
$$\frac{x_1^{(k)}+x_2^{(k)}+\ldots +x_{a+m-1}^{(k)}}{m-1}\leq x_{a+m}^{(k)}.$$
Therefore we may write
$$\frac{x_1^{(k)}+x_2^{(k)}+\ldots +x_{a+m-1}^{(k)} + x_{a+m}^{(k)}}{m} = \frac{x_1^{(k)}+x_2^{(k)}+\ldots +x_{a+m-1}^{(k)}}{m-1}\cdot \frac{m-1}{m} + \frac{x_{a+m}^{(k)}}{m}\leq$$
$$ \leq x_{a+m}^{(k)} \left(\frac{m-1}{m}+\frac{1}{m}\right) = x_{a+m}^{(k)}\leq x_{a+m+1}^{(k)}.$$ 
Finally we get
$$\frac{x_1^{(k)}+x_2^{(k)}+\ldots +x_{a+b}^{(k)}}{b}\leq x_{a+b}^{(k)},$$
which means $x^{(k)}\in A$. \end{proof}

As far as the set $D$ is concerned, we may strengthen the conclusion of Lemma \ref{sigma+}.

\begin{lemma} Let $a\geq 1$ and $b\geq 2$. For almost every $x\in D$ we have $$\lim_{k\to\infty} \tab^k(x) = (0,\ldots,0,x_{a+2}^{\infty},\ldots,x_{a+b}^{\infty}),$$ where $x_{j}^{\infty} >0$, $j=a+2,\ldots,a+b$.
\end{lemma}

\begin{proof} The property $x_{a+2} \geq  x_1+\ldots+x_{a+1}$  implies that for any $x\in D$ we have
$$T_{a,b}(x)=(T_{a,1}(x_1,\ldots,x_{a+1}),x_{a+2}-x_a,\ldots,x_{a+b}-x_a).$$ Since $T(D)\subset D$, iterating the previous formula gives 
\begin{equation} \label{2} T_{a,b}^k(x)=(T_{a,1}^k(x_1,\ldots,x_{a+1}),x_{a+2}-\sum_{i=0}^{k-1} x_a^{(i)},\ldots,x_{a+b}-\sum_{i=0}^{k-1} x_a^{(i)}),\end{equation}
for any $x\in D$ and $k\geq 1$. Let now $x\in D$ satisfy $T_{a+1}^k(x_1,\ldots,x_{a+1})\to 0$ and $x_{a+2}>x_1+\ldots+x_{a+1}$. Corollary \ref{brun} implies that this is true for almost every $x\in D$ and also that $x_{j}^{\infty}=x_j-\sigma(x_1,\ldots,x_{a+1})>0$, $j=a+2,\ldots,a+b$. \end{proof}


\section{Dynamics outside $A$}

In order to show that $A$ is an absorbing set we have to understand the dynamics of $T_{a,b}$ on its complement. Define
$$ cA= \Lambda^{a+b}\setminus A=\{x\in\Lambda^{a+b} : \sigma(x) > bx_{a+b} \}.$$ 
We begin with the following remark.

\begin{lemma} Let $x\in {cA}$ and suppose that $x_{a+b}\geq 2x_a$. Then $T_{a,b}(x)\in {c A}$.
\end{lemma}

\begin{proof} Let $x\in\Lambda^{a+b}$ and write $y=T_{a,b}(x)$. We have $y_1=\min\{x_1,x_{a+1}-x_a\}$ and $y_{a+b}=\max\{x_a,x_{a+b}-x_a\}$. If we suppose that $x\in {c A}$ and $x_{a+b}\geq 2x_a$, then $y_{a+b}=x_{a+b}-x_a$. We get
$$ by_{a+b}=bx_{a+b}-bx_a < \sigma(x)-bx_a=\sigma(y),$$ 
which means $y\in {c A}$. \end{proof}

Having this in mind, we define the following subset of ${c A}$ 
$$ \Theta = \{x\in {c A}: 2x_a\geq x_{a+b} \}.$$ 
The previous lemma implies that any orbit starting in ${c A}$ and eventually attracted by $A$ has to visit $\Theta$. First, we will show that almost every orbit starting in ${cA}$ visits $\Theta$ (which does not mean it is attracted by $A$ though). 

\begin{lemma} For almost every $x\in {cA}$ there exists $k\geq 0$ such that $T^k_{a,b}(x)\in\Theta$. 
\end{lemma} 

\begin{proof} Let $x\in {cA}$ be a point whose orbit never leaves ${cA}$. We know (Proposition \ref{origin}) that in this case $x^{(k)}$ converges to the origin almost surely, which also means (Lemma \ref{sum}) that $\sigma(x)=b\sum_{k\geq 0} x_a^{(k)}$. If the orbit of $x$ never visits $\Theta$, for every $k\geq 0$ we have $x_{a+b}^{(k+1)}=x_{a+b}^{(k)}-x_a^{(k)}$, which implies $x_{a+b}= \sum_{k\geq 0} x_a^{(k)}=\sigma(x)/b$. This contradicts the definition of ${cA}$. \end{proof}

For convenience we set
$$\Gamma=cA\cap T^{-1}(A) \subset \Theta.$$
To show that $A$ is absorbing for $\tab$, it is enough to show that the set $\Gamma$ absorbs almost every orbit starting in ${cA}$. To this end, we consider a kind of a first return map $P:\Theta\to\Theta$, conditioned on the set $\Gamma$:
$$P(x) = \left\{ \begin{array}{ll}  x & \text{if} \ x\in \Gamma \\ T_{a,b}(x) & \text{if} \ x\not\in \Gamma \ \text{and} \ T_{a,b}(x)\in\Theta \\ T_{a,b}^k(x) & \text{if} \ x \not\in \Gamma, \  T_{a,b}^k(x)\in\Theta \ \text{and} \ T_{a,b}^{i}(x)\not\in\Theta, i=1,\ldots,k-1. \end{array}\right. $$ 
This map is well defined for almost every point in $\Theta$. 

The map $P$ is piecewise linear just as $T_{a,b}$ is, but now the number of cylinders is infinite, since the first return time to $\Theta$ is not bounded. Let us define them explicitly. For every $k\geq 1$ and every set of permutations $\pi_1,\ldots,\pi_k\in\Pi_{a,b}$, we put
$$ \Theta_{\pi_1,\ldots,\pi_k} = \{x\in \Theta\setminus \Gamma : P(x)=\tab^k(x) \ \text{and} \ \tab^i(x)= L_{\pi_i}\cdots L_{\pi_1}x^T \ \text{for} \ i=1,\ldots,k \}. $$
Some of these cylinders are empty. If a given cylinder is nonempty, then it corresponds to the set of points of $\Theta$, whose orbits under $\tab$ visit the sequence of cylinders $\Lambda_{\pi_i}$, $i=1,\ldots,k$, before returning to the set $\Theta$.  On such a  cylinder, $P$ is given by the matrix $L_{\pi_1,\ldots,\pi_k} = L_{\pi_k}\cdots L_{\pi_1}$ which is a product of a finite number of matrices $L_{\pi_i}$, corresponding to the cylinders of $T_{a,b}$ visited by the orbit. We may thus decompose $\Theta$ as follows
\begin{equation} \label{partition} \Theta=\Gamma \cup \bigcup_{k=1}^{\infty} \bigcup_{(\pi_1,\ldots,\pi_k)\in\Pi_{a,b}^k} \Theta_{\pi_1,\ldots,\pi_k}.\end{equation}

The cylinders have the following property. 

\begin{prop} \label{part-full} For every choice $\pi_1,\ldots,\pi_k\in\Pi_{a,b}$ the corresponding cylinder $\Theta_{\pi_1,\ldots\pi_k}$ is either empty or full for the map $P$, i.e. $P(\Theta_{\pi_1,\ldots,\pi_k})=\Theta$ whenever $\Theta_{\pi_1,\ldots,\pi_k}$ is nonempty. \end{prop}

\begin{proof} The set $\Theta$ is defined by the inequality $x_a\geq x_{a+b}-x_a$, where both terms are the coordinates of $\tab(x)$. Since every cylinder $\Lambda_{\pi}$ corresponds to some order of coordinates of $\tab(x)$, for every $\pi\in\Pi_{a,b}$ we have either $\Lambda_{\pi}\cap {cA}\subset \Theta$ or $\Lambda_{\pi}\cap {cA}\subset  {cA}\setminus \Theta$. We have the following:
\begin{equation} \label{image}  P(\Theta_{\pi_1,\ldots,\pi_k})=\tab^k\{x\in\Delta_{\pi_1}\setminus \Gamma : x^{(1)}\in\Delta_{\pi_2}, \ldots,x^{(k-1)}\in\Delta_{\pi_k},x^{(k)}\in\Theta\},\end{equation}  where all the cylinders involved are restricted to ${cA}$. It is obvious that our cylinder is empty, unless $\Lambda_{\pi_1}\cap {cA}\subset \Theta$ and $\Lambda_{\pi_i}\cap {cA}\subset {c\Theta}$ for $i=2,\ldots,k$.

Suppose that our cylinder is nonempty. Recall that the cylinders $\Lambda_{\pi}$ corresponding to the map $\tab$ are full. Since $\tab(A)\subset A$ and $\tab(\Gamma)\subset A$, we must have $\tab(\Lambda_{\pi}\cap({cA}\setminus \Gamma))={cA}$ for every $\pi\in\Pi_{a,b}$. This, together with (\ref{image}), clearly implies that $P(\Theta_{\pi_1,\ldots,\pi_k})=\Theta$. \end{proof}

Let us slightly alter the notations in order to introduce the sequence of refinements of the partition (\ref{partition}). Since this partition corresponds to the first iterate of the map $P$ (or the first return time of $\tab$ to $\Theta$), we will call its cylinders $\Theta^{(1)}_{\nu}$, where $\nu$ is some finite sequence of permutations from the set $\Pi_{a,b}$.  We may rewrite (\ref{partition}) as 
$$\Theta=\Gamma\cup\bigcup_{k=1}^{\infty} \bigcup_{\nu\in\Pi_{a,b}^k} \Theta^{(1)}_{\nu}.$$

Every cylinder $\Theta^{(1)}_{\sigma}$ may be partitioned further, considering the second iteration of $P$. For every $k\geq 1$ and $\nu' \in\Pi_{a,b}^k$ we define
$$\Theta^{(2)}_{\nu\star\nu'} = \{x\in\Theta^{(1)}_{\nu} \ : \ P(x)\in \Theta^{(1)}_{\nu'} \},$$  which generates the following refinement of the partition (\ref{partition}):
$$\Theta=P^{-1}(\Gamma)\cup \bigcup_{k,m=1}^{\infty} \bigcup_{\nu\in\Pi_{a,b}^k}  \bigcup_{\nu'\in\Pi_{a,b}^m} \Theta^{(2)}_{\nu\star \nu'}.$$
It is not hard to show that the cylinders of this new, finer partition are either empty or full for the map $P^2$. 

Following the same pattern, for every $n\geq 1$ we may consider the $n$th iteration of the map $P$ and define the corresponding $n$th partition of $\Theta$:
\begin{equation} \label{partition-n} \Theta = P^{-(n-1)}(\Gamma)\cup \bigcup_{k_1,k_2,\ldots,k_n=1}^{\infty} \bigcup_{\nu_1\in\Pi_{a,b}^{k_1}} \ldots  \bigcup_{\nu_n\in\Pi_{a,b}^{k_n}} \Theta^{(n)}_{\nu_1\star\ldots\star\nu_n}, \end{equation}  
where 
$$ \Theta^{(n)}_{\nu_1\star\ldots\star\nu_n} = \{x\in \Theta^{(1)}_{\nu_1} \ : \ P(x)\in \Theta^{(1)}_{\nu_2},\ldots,P^{n-1}(x)\in\Theta^{(1)}_{\nu_n}\}.$$
The following result may be deduced by induction from Proposition \ref{part-full} and the definition of the $n$th partition.

\begin{prop} For every $k_1,\ldots,k_n\geq 1$ and every $\nu_i\in\Pi_{a,b}^{k_i}$, $i=1,\ldots,n$, the corresponding cylinder $\Theta^{(n)}_{\nu_1\star\ldots\star\nu_n}$ is either empty or full for the map $P^n$, i.e. $P^n(\Theta^{(n)}_{\nu_1\star\ldots\star\nu_n})=\Theta$ whenever $\Theta^{(n)}_{\nu_1\star\ldots\star\nu_n}$ is nonempty.  \end{prop} 


\section{Projected dynamics} 

In this section, we will show that some constant proportion of points from every cylinder $\Theta^{(n)}_{\nu_1\star\ldots\star\nu_n}$ falls into $\Gamma$ under the action of $P^n$. To make this statement meaningful, we project our dynamics onto the simplex $\Delta=\{x\in\Lambda^{a+b} \ : \ \sigma(x)=1\}$, which is of  finite $(a+b-1)$-dimensional Lebesgue measure. Let $p:\Lambda^{a+b}\to \Delta$ be the projection given by $x\mapsto x/\sigma(x)$ and consider the new transformation $\tilde{T}_{a,b} :\Delta \to \Delta$ which makes the following diagram commute:

$$\begin{CD}
\mbox{$\Lambda^{a+b}$} @>\mbox{$\tab$}>> \mbox{$\Lambda^{a+b}$}\\
@VpVV @VVpV\\
\mbox{$\Delta$} @>>\mbox{$\tilde{T}_{a,b}$}> \mbox{$\Delta$}
\end{CD}$$

We may define the projected counterparts of the sets $A$ and $D$ which are given by 
$$\tilde{A}=\{ x\in \Delta  :  1>bx_{a+b} \} \ \ \text{and} \ \  \tilde{D}=\{x\in\Delta  : x_1+x_2+\ldots+x_{a+1}\leq x_{a+2}\}.$$ 
It is easy to see that both sets are $\tilde{T}_{a,b}$-forward invariant. We will also consider the projected set 
\begin{equation} \label{theta} \tilde{\Theta}=\{ x\in\Delta  : 1 > bx_{a+b}, \ 2x_a\geq x_{a+b}\}   \end{equation} 
with its subset $\tilde{\Gamma}$, the corresponding first return map $\tilde{P} : \tilde{\Theta}\to\tilde{\Theta}$ and the underlying partitions into cylinders $\tilde{\Theta}^{(n)}_{\nu_1\star\ldots\star\nu_n}$. 

\begin{prop} \label{ratio} 
There exists a constant $\alpha\in(0,1)$ such that for every $n\geq 1$ and every projected cylinder $\tilde{\Theta}_{\nu_1\star\ldots\star\nu_n}^{(n)}$ of the $n$th partition of $\tilde{\Theta}$,  we have
\begin{equation} \label{rapport} \frac{Leb(x\in\tilde{\Theta}_{\nu_1\star\ldots\star\nu_n}^{(n)} : \tilde{P}^n(x)\in \tilde{\Gamma})}{Leb (\tilde{\Theta}_{\nu_1\star\ldots\star\nu_n}^{(n)})}\geq \alpha,\end{equation} 
where $Leb$ stands for $(a+b-1)$-dimensional Lebesgue measure in $\Delta$. 
\end{prop}

\begin{proof} 
Let $\tilde{\Theta}_{\nu_1\star\ldots\star\nu_n}^{(n)}$ be a cylinder. As the initial cylinder ${\Theta}_{\nu_1\star\ldots\star\nu_n}^{(n)}$ is full for the map $P^n$, we easily get that $\tilde{P}^n(\tilde{\Theta}_{\nu_1\star\ldots\star\nu_n}^{(n)})=\tilde{\Theta}$. Moreover, on $\tilde{\Theta}_{\nu_1\star\ldots\star\nu_n}^{(n)}$ the map $\tilde{P}^n$ is defined by 
\begin{equation} \label{P} \tilde{P}^n(x)=\frac{Lx^T}{\sigma(Lx^T)}, \end{equation}
where $L$ is the finite product of matrices $L_{\pi_i}$ corresponding to the cylinders of the original transformation $\tab$ (the projection map $p:\Lambda^{a+b}\to\Delta$ preserves the order of coordinates at any point of $\Lambda^{a+b}$). Let $M=L^{-1}$ and denote by $\tilde{P}_M^{-n}:\tilde{\Theta}\to\tilde{\Theta}_{\nu_1\star\ldots\star\nu_n}^{(n)}$ the corresponding inverse branch of $\tilde{P}^n$. Once again we have 
\begin{equation} \label{PB}  \tilde{P}_M^{-n}(x)=\frac{Mx^T}{\sigma(Mx^T)}.\end{equation}
We get 
$$Leb (\tilde{\Theta}_{\nu_1\star\ldots\star\nu_n}^{(n)})=\int_{\tilde{\Theta}_{\nu_1\star\ldots\star\nu_n}^{(n)}} dx=\int_{\tilde{P}_M^{-n}(\tilde{\Theta})}dx=\int_{\tilde{\Theta}} J(\tilde{P}_M^{-n}) dx$$
and 
$$Leb(x\in\tilde{\Theta}_{\nu_1\star\ldots\star\nu_n}^{(n)} : \tilde{P}^n(x)\in \tilde{\Gamma})=\int_{\tilde{P}_M^{-n}(\tilde{G})}dx=\int_{\tilde{\Gamma}} J(\tilde{P}_M^{-n}) dx,$$
where $J(\tilde{P}_M^{-n})$ stands for the Jacobian of the transformation defined by (\ref{PB}) and $dx$ is $(a+b-1)$-dimensional Lebesgue volume element on $\Delta$. The matrix $M$ is non-negative and $\vert \det(M)\vert =1$. One may check (cf. \cite{arnaldo5}), that in this case the corresponding Jacobian is given by

$$J(\tilde{P}_M^{-n})(x)=\frac{1}{K(c_1x_1+c_2x_2+\ldots+c_{a+b}x_{a+b})^{a+b}},$$
where $c_i$ is the sum of terms over the $i$th column of $M$ and $K>0$ is a constant depending on the dimension $a+b$ but independent of the matrix $M$. 

In order to estimate the integrals above, we have to estimate the Jacobian on the set $\tilde{\Theta}$. Let $x\in \tilde{\Theta}$. Since its coordinates are ordered in the nondecreasing order and their sum is equal to one, we must have $x_{a+b}\geq 1/(a+b)$. Together with (\ref{theta}), we get 
\begin{equation} \frac{1}{2(a+b)} \leq x_a\leq x_{a+1}\leq \ldots\leq x_{a+b}\leq \frac{1}{b}\leq 1.\end{equation} 
This implies that on the set $\tilde{\Theta}$ we have
\begin{equation} \label{jac} \frac{1}{K(c_1+c_2+\ldots+c_{a+b})^{a+b}}\leq J(\tilde{P}^{-n}_M)\leq \frac{2^{a+b}(a+b)^{a+b}}{K(c_a+c_{a+1}+\ldots+c_{a+b})^{a+b}}. \end{equation}

We claim that
\begin{equation} \label{claim} c_1+c_2+\ldots+c_{a-1} \leq (a-1) (c_a+c_{a+1}+\ldots + c_{a+b}).\end{equation}
In order to prove this assertion it is enough to show that
$$\max_{1\leq i\leq a+b} c_i=\max_{a\leq i\leq a+b} c_i.$$
Recall that the matrix $M$ is a finite product of matrices $M_{\pi}$ whose column vectors are described by (\ref{columns}). We will proceed by induction on the number of matrices involved in the product. If $M=M_{\pi_1}$ is just one of the elementary matrices, the claim follows directly from (\ref{columns}). Now suppose that our inequality is true for any product of elementary matrices up to the length $n$ and let $M=M_{\pi_1}M_{\pi_2}\cdots M_{\pi_n} M_{\pi_{n+1}}$.  Let $c_i^{(n+1)}$ be the sum of terms over the $i$th column of $M$ and $c_i^{(n)}$ the sum over the $i$th column of $M_{\pi_1}M_{\pi_2}\cdots M_{\pi_n}$. Once again from (\ref{columns}), we get that there exists $a\leq i\leq a+b$, such that
$$ c_i^{(n+1)}=c_a^{(n)}+c_{a+1}^{(n)}+\ldots+ c_{a+b}^{(n)}\geq \max_{a\leq k\leq a+b}c_k^{(n)}=\max_{1\leq k\leq a+b}c_k^{(n)}.$$ Moreover, for every $j\neq i$ there exists $1\leq k\leq a+b$ such that $c_j^{(n+1)}=c_k^{(n)}$. This easily implies (\ref{claim}). 

Let us complete the proof of Proposition \ref{ratio}. The inequalities (\ref{jac}) and (\ref{claim}) imply together
$$   \frac{1}{Ka^{a+b} (c_a+c_{a+1}+\ldots+c_{a+b})^{a+b}}\leq J(\tilde{P}^{-n}_M)\leq \frac{2^{a+b}(a+b)^{a+b}}{K(c_a+c_{a+1}+\ldots+c_{a+b})^{a+b}} $$
on $\tilde{\Theta}$. Thanks to this estimation we may write
$$ \frac{Leb(x\in\tilde{\Theta}_{\nu_1\star\ldots\star\nu_n}^{(n)} : \tilde{P}^n(x)\in \tilde{\Gamma})}{Leb (\tilde{\Theta}_{\nu_1\star\ldots\star\nu_n}^{(n)})} = \frac{\int_{\tilde{\Gamma}} J(\tilde{P}^{-n}_M)dx}{\int_{\tilde{\Theta}} J(\tilde{P}^{-n}_M)dx}\geq \left(\frac{1}{2a(a+b)}\right)^{a+b}\frac{Leb(\tilde{\Gamma})}{Leb(\tilde{\Theta})}.$$ 

\end{proof}


\section{Proofs of main results}

We are now ready to prove our main results.

\begin{proof}[Proof of Theorem \ref{A}] In order to show that the set $A$ is absorbing for the map $\tab$, it is enough to show that the set $\Gamma$ is absorbing for the first return map $P:\Theta\to\Theta$. We will rather work in the projected space and show that the projected set $\tilde{\Gamma}$ is absorbing for the projected first return map $\tilde{P}$. It is sufficient, since the projection map $p:\Lambda^{a+b}\to \Delta$ sends any set of positive $(a+b)$-dimensional Lebesgue measure in $\Lambda^{a+b}$ onto some set of positive $(a+b-1)$-dimensional Lebesgue measure in $\Delta$.  

We may rewrite the decomposition (\ref{partition-n}) of $\Theta$ in the projected version
$$\tilde{\Theta}=\tilde{P}^{-(n-1)}(\tilde{\Gamma})\cup \bigcup_{k_1,k_2,\ldots,k_n=1}^{\infty} \bigcup_{\nu_1\in\Pi_{a,b}^{k_1}} \ldots  \bigcup_{\nu_n\in\Pi_{a,b}^{k_n}} \tilde{\Theta}^{(n)}_{\nu_1\star\ldots\star\nu_n}.$$
From (\ref{rapport}) we deduce that 
$$Leb(\tilde{\Theta}\setminus \tilde{P}^{-n}(\tilde{\Gamma})) \leq (1-\alpha) Leb(\tilde{\Theta}\setminus \tilde{P}^{-(n-1)}(\tilde{\Gamma})),$$
for every $n\geq 1$. Since $\alpha\in(0,1)$ is a constant independent of $n$, this clearly implies that $\tilde{\Gamma}$ is absorbing for $\tilde{P}$.  \end{proof}

Combining Theorem \ref{A} and Lemma \ref{sigma+}, we prove Theorem \ref{limit} which gives a precise description of the limiting behavior of the orbits.

\begin{proof}[Proof of Theorem \ref{limit}] Fix $a\geq 1$ and $b\geq 2$. Since the set $A$ is absorbing for the map $\tab$, it is enough to show that 
$$
\lim_{k\to\infty}\tab^k(x)=(0,\ldots,0,x_{a+2}^{\infty},\ldots,x_{a+b}^{\infty}),
$$
with $x_{a+2}^{\infty}>0$ for almost every $x\in A$. From Lemma \ref{sigma+} we know that $x_{a+b}^{\infty}>0$ almost surely. If $b=2$ we are done.

If $b\geq 3$ we still have $x_{a+b}^{\infty}>0$ for almost every $x\in A$. Define the  exceptional set $E_{b-1}=\{ x\in A \ : \ x_{a+b-1}^{\infty}=0\}$ and let $x\in E_{b-1}$. It is easy to see that there exists some $k_0\geq 1$ such that for every $k\geq k_0$ we have 
$$ 
x^{(k+1)}=(T_{a,b-1}(x_1^{(k)},\ldots,x_{a+b-1}^{(k)}), x_{a+b}^{(k)}-x_a^{(k)}).
$$
In other terms, starting from some $k_0$th iteration, the last coordinate is the largest one along the whole $\tab$-orbit of $x$. Since $x\in E_{b-1}$, we get
$$
\lim_{k\to\infty} x^{(k_0+k)}=\lim_{k\to\infty} (T_{a,b-1}^{k}(x_1^{(k_0)},\ldots,x_{a+b-1}^{(k_0)}), x_{a+b}^{(k_0)}-\sum_{j=1}^{k-1}x_a^{(k_0+j)})=(0,\ldots,0,x_{a+b}^{\infty}).
$$ 

Define $\Sigma=\{x\in \Lambda^{a+b-1} \ : \ T_{a,b-1}^k(x)\to (0,\ldots,0)\}$. Once again, from Theorem \ref{A} and Lemma \ref{sigma+} we deduce that $\Sigma$ is of zero Lebesgue measure in $\Lambda^{a+b-1}$. The analysis above shows that if $x \in E_{b-1}\subset \Lambda^{a+b}$, then there exists $k_0\geq 1$ such that $(x_1^{(k_0)},\ldots,x_{a+b-1}^{(k_0)})\in\Sigma$. We may write
$$E_{b-1}\subset \bigcup_{k_0=1}^{\infty} T_{a,b}^{-k_0}(\Sigma \times \R_+).$$
Since $\tab$ restricted to any cylinder preserves Lebesgue measure, by Fubini Theorem we get that $E_{b-1}$ is a set of zero measure in $\Lambda^{a+b}$. Equivalently, for almost every $x\in \Lambda^{a+b}$ we have $x_{a+b-1}^{\infty}>0$. If $b=3$ we are done. 

If $b>3$ we may repeat the argument above to show that all the exceptional sets 
$E_{b-i}= \{ x\in A \ : \ x_{a+b-i}^{\infty}=0\}$, $i=2, \ldots, b-2,$ are null measure sets.  \end{proof} 

It is clear from Theorem \ref{limit} that for every $a\geq 1$ and $b\geq 2$ the set $D$ is absorbing for the map $\tab$. Obviously, this is also the case for the projected Schweiger map $S_{a,b}:B\to B$ defined in Section 1. 
\\

Next we will prove Theorem \ref{erg}, which together with Theorem \ref{noterg} give our partial answer to the second Schweiger conjecture. To this end, we recall a result from \cite{arnaldo1}.

\begin{thm}[\cite{arnaldo1}, Corollary 8.2] \label{euclid}
The map $T_{1,1}$  is ergodic with respect to Lebesgue measure.
\end{thm} 

We use it to show that $S_{1,2}:B\to B$ is ergodic.

\begin{proof}[Proof of Theorem \ref{erg}] Recall that the transformation $S_{1,2}$ is defined as a projection of the map $T_{1,2}:\Lambda^{3}\to\Lambda^{3}$ on the set $B=\{x\in\Lambda^{2} \ : \  x_{2}\leq 1\}$. On the projected set $D=\{ x\in B \ : \ x_1+x_{2}\leq 1\}$, which is absorbing for $S_{1,2}$, we have $T_{1,2}(x)=(T_{1,1}(x_1,x_{2}), x_{3}-x_1)$. Thus
$$S_{1,2}(x)=\frac{1}{1-x_1}T_{1,1}(x).$$
Moreover, as $k\to\infty$, for almost every $x\in D$ we have $T_{1,2}^k(x)\to (0,0,x_{3}^{\infty})$ where $x_{3}^{\infty}>0$. This implies that for almost every $x\in D$ the corresponding $S_{1,2}$-orbit converges to the origin. We get that $S_{1,2}$ is totally dissipative with respect to Lebesgue measure. 

Now consider the map $\phi:\Lambda^2 \to D$ defined by  
$$\phi(x_1,x_2)=\left(\frac{x_1}{1+x_1+x_2},\frac{x_2}{1+x_1+x_2}\right).$$ It is invertible and sends $\Lambda^2$ onto the set $\{x\in D, x_1+x_2<1\}$ of full Lebesgue measure in $D$. Moreover, $\phi$ and $\phi^{-1}$ send null measure sets to null measure sets. One may check that 
$$T_{1,1} = \phi^{-1}\circ S_{1,2}\circ \phi.$$
This together with Theorem \ref{euclid} imply the ergodicity of $S_{1,2}$ on the set $D$. Since $D$ is almost surely absorbing, this also implies the ergodicity on the larger set $B$.   \end{proof}

Just as the map $S_{1,2}$ is related to the Euclidean algorithm, for $a\geq 2$ the dynamics of the transformation $S_{a,2}:D\to D$ is related to the Brun algorithm. Indeed, the subset $\Delta= \{x\in D : \sigma(x)=1\}$ is invariant under $S_{a,2}$ and its dynamics restricted to this set coincides with the projection of the Brun algorithm $T_{a,1}$. Our transformation $S_{a,2}$ may be thus seen as an extension of the ergodic dynamics of the Brun algorithm on $\Delta$ to the larger set $\{x : \sigma(x)\leq 1\}$. It is thus natural to conjecture that its dynamics is ergodic also in the case $a\geq 2$.

To show that $S_{a,b}:B\to B$ is not ergodic for $b\geq 3$ we use the limiting behavior of orbits to exhibit an invariant nonconstant function.  

\begin{proof}[Proof of Theorem \ref{noterg}] For every $x\in B$ we have 
$$\lim_{k\to\infty} S_{a,b}^k(x)=(x_1^{\infty},\ldots, x_{a+b-1}^{\infty}).$$
Define $f:B\to \R$ by $f(x)=x_{a+b-1}^{\infty}$. It is measurable and $S_{a,b}$-invariant. For almost every $x\in D$, an argument similar to one used in the proof of Lemma 4.7 gives that 
$$x_{a+b-1}^{\infty}=\frac{x_{a+b-1}-(x_1+\ldots+x_{a+1})}{1-(x_1+\ldots+x_{a+1})}.$$
This shows that $f$ is nonconstant. The map $S_{a,b}$ is not ergodic for $b\geq 3$. \end{proof} 


\section{Other subtractive algorithms} Fix $a\geq 2$, $b\geq 1$ and $1\leq i\leq a-1$.  We may alter slightly the definition of the transformation $\tab$ by subtracting the coordinate $x_i$ instead of $x_a$. The new transformation $T:\Lambda^{a+b}\to\Lambda^{a+b}$ is given by the formula
$$ T(x_1,x_2,\ldots,x_{a+b})=\pi(x_1,\ldots,x_a,x_{a+1}-x_i,\ldots,x_{a+b}-x_i),$$
where $\pi$ is a permutation which arranges the coordinates in ascending order. This new family of transformations also contains some well studied ones. For example, if $i=1$ and $b=1$ we get the Selmer algorithm. 

We may remark the following properties of the set $A$ with respect to the transformation $T$ defined above.

\begin{lemma} 
\begin{itemize}
\item[(1)] For every $a\geq 2$, $b\geq 1$ and $1\leq i \leq a-1$ we have $T(A)\subset A$.
\item [(2)] If $b \leq a+1-i$ then $T^{-1}(A)\subset A$.
\end{itemize}
\end{lemma}

\begin{proof}
\begin{itemize}
\item[(1)] Let $x\in A$ and $y=T(x)$. We have
$$by_{a+b} = b\max\{x_a,x_{a+b}-x_i\} \geq b(x_{a+b}-x_i) \geq  \sigma(x)-bx_i = \sigma(y).$$
This implies $y\in A$. 

\item[(2)] Let $x\in ^cA$ and $y=T(x)$. Suppose that $b \leq a+1-i$ . We have 
$$
b(x_{a+b}-x_i)\leq \sigma(x)-bx_i=\sigma(y)
$$ 
and 
$$
bx_a\leq \sum_{j=a+1}^{a+b}x_j\leq \sum_{j=1}^{i-1}x_j+\sum_{j=i}^{a}(x_j-x_i)+\sum_{j=a+1}^{a+b}x_j\leq \sum_{j=1}^{a+b}x_j -bx_i=\sigma(y).
$$ 
Together this gives $by_{a+b}=b\max\{x_a,x_{a+b}-x_i\}\leq \sigma(y)$.
\end{itemize}

\end{proof}

In virtue of the last lemma, when looking for the globally absorbing set one has to assume that $$b\geq a+3-i.$$ Under this assumption, we may ask if $A$ is absorbing for the new algorithm $T$. The difficulty of this question comes from the fact that the cylinders for the map $T$ are not full as it was the case for $\tab$. This is the main reason why the argument developed in Section 5 does not work in this case. We are no longer able to estimate the proportion of points absorbed by $A$ from every cylinder. 
  
However, it is rather straightforward to check that the arguments presented in Section 3, used to study limiting properties of orbits, are still valid for the new transformation $T$. 

\begin{prop} Let $a\geq 2$, $b\geq 1$ and $1\leq i\leq a-1$. 

\begin{itemize}
\item[(1)] For almost every $x\in \Lambda^{a+b}$ the corresponding orbit converges to $(x_1^{\infty}, x_{2}^{\infty},\ldots,x_{a+b}^{\infty})$ with $x_1^{\infty}=\ldots=x_{a+1}^{\infty}=0$. 
\item[(2)] For almost every $x\in A$ we have $x_{a+b}^{\infty}>0$.
\item[(3)] For almost every $x\in\Lambda^{a+b}$ whose orbit does not converge to the origin, there exists $k\geq 1$ such that $T^k(x)\in A$.
\end{itemize}
\end{prop}

\bibliographystyle{plain}

\bibliography{hsa3}

\begin{thebibliography}{10}

\bibitem{avila}
A.~Avila and G.~Forni.
\newblock Weak mixing for interval exchange transformations and translation
  flows.
\newblock {\em Ann. of Math. (2)}, 165(2):637--664, 2007.

\bibitem{fkn}
R.~Fokkink, C.~Kraaikamp, and H.~Nakada.
\newblock On schweiger's problems on fully subtractive algorithms.
\newblock {\em to appear in Israel J. Math.}

\bibitem{kraaikamp-meester}
C.~Kraaikamp and R.~Meester.
\newblock Ergodic properties of a dynamical system arising from percolation
  theory.
\newblock {\em Ergodic Theory Dynam. Systems}, 15(4):653--661, 1995.

\bibitem{levitt}
G.~Levitt.
\newblock La dynamique des pseudogroupes de rotations.
\newblock {\em Invent. Math.}, 113(3):633--670, 1993.

\bibitem{meester-nowicki}
R.~Meester and T.~Nowicki.
\newblock Infinite clusters and critical values in two-dimensional circle
  percolation.
\newblock {\em Israel J. Math.}, 68(1):63--81, 1989.

\bibitem{arnaldo5}
A.~Messaoudi, A.~Nogueira, and F.~Schweiger.
\newblock Ergodic properties of triangle partitions.
\newblock {\em Monatsch. Math.}, 157(3):283--299, 2009.

\bibitem{arnaldo1}
A.~Nogueira.
\newblock The three-dimensional {P}oincar\'e continued fraction algorithm.
\newblock {\em Israel J. Math.}, 90(1-3):373--401, 1995.

\bibitem{arnaldo2}
A.~Nogueira and D.~Rudolph.
\newblock Topological weak-mixing of interval exchange maps.
\newblock {\em Ergodic Theory Dynam. Systems}, 17(5):1183--1209, 1997.

\bibitem{rauzy}
G.~Rauzy.
\newblock \'{E}changes d'intervalles et transformations induites.
\newblock {\em Acta Arith.}, 34(4):315--328, 1979.

\bibitem{schweiger}
F.~Schweiger.
\newblock {\em Multidimensional continued fractions}.
\newblock Oxford Science Publications. Oxford University Press, Oxford, 2000.

\bibitem{veech1}
W.~A. Veech.
\newblock Gauss measures for transformations on the space of interval exchange
  maps.
\newblock {\em Ann. of Math. (2)}, 115(1):201--242, 1982.

\bibitem{veech2}
W.~A. Veech.
\newblock The metric theory of interval exchange transformations. {III}. {T}he
  {S}ah-{A}rnoux-{F}athi invariant.
\newblock {\em Amer. J. Math.}, 106(6):1389--1422, 1984.

\end{thebibliography}

\vskip1cm
\begin{flushleft}
Tomasz Miernowski \;\; and \;\; Arnaldo Nogueira\\
Institut de Math\'ematiques de Luminy\\
163, avenue de Luminy, Case 907\\
13288 Marseille Cedex 9, France\\

\smallskip
E-mail: {\tt miernow@iml.univ-mrs.fr\;\;} and {\;\;\tt nogueira@iml.univ-mrs.fr}

\end{flushleft}

\end{document}